\documentclass[11pt]{amsart}
\usepackage{amsthm}
\usepackage{array}
\usepackage{amsmath}
\usepackage{enumerate}
\usepackage{tikz}
\usetikzlibrary{calc}
\usepackage{amssymb}
\usepackage{latexsym}
\usepackage{amsfonts}
\usepackage{color}
\usepackage{mathrsfs}
\usepackage{epsfig,helvet}

\theoremstyle{plain} \numberwithin{equation}{section}
\newtheorem{thm}{Theorem}[section]
\newtheorem{theorem}[thm]{Theorem}
\newtheorem{lemma}[thm]{Lemma}
\newtheorem{corollary}[thm]{Corollary}
\newtheorem{example}[thm]{Example}
\newtheorem{definition}[thm]{Definition}

\begin{document}
\setcounter{page}{1}

\title[Some properties of isoclinism in Lie superalgebras]{Some properties of isoclinism in Lie superalgebras}

\author[Nayak]{Saudamini Nayak}
\address{Harish-Chandra Research Institute (HBNI) \\
         Chhatnag Road, Jhunsi, 
          Allahabad-211 019 \\
                India}
\email{ anumama.nayak07@gmail.com}

\thanks{Research supported by the Harish-Chandra Research Institute Postdoctoral Fellowship.}
\author[Padhan]{Rudra Narayan Padhan}
\address{Department of Mathematics, National Institute of Technology  \\
         Rourkela,
          Odisha-769028 \\
                India}
\email{rudra.padhan6@gmail.com}

\author[Pati]{Kishor Chandra Pati}
\address{Department of Mathematics, National Institute of Technology  \\
         Rourkela, 
          Odisha-769028 \\
                India}
\email{kcpati@nitrkl.ac.in}

\maketitle
\begin{abstract}
Isoclinism of Lie superalgebras has been defined and studied currently. In this article it is shown that for finite dimensional Lie superalgebras of same dimension, the notation of isoclinism and isomorphism are equivalent. Furthermore we show that covers of finite dimensional Lie superalgebras are isomorphic using isoclinism concept. 
\end{abstract}
\section{Introduction}
In 1940, P. Hall  introduced an equivalence relation on the class of all groups called isoclinism, which is weaker than isomorphism and plays an important role in classification of finite $p$-groups \cite{Kar1987}. In 1994, K. Moneyhun \cite{Moneyhun1994, MoneyhunD1994} gave a Lie algebra analogue of the concept of isoclinism. Further Saeedi and Veisi \cite{Saeedi2014} have defined the same notation for $n$-Lie algebras. Similarly, isoclinism has been defined and studied for Lie superalgebras recently \cite{Nayak2018}.
\begin{definition}
Let $L$ and $K$ be two Lie superalgebras, $\varphi: \frac{L}{Z(L)}\longrightarrow \frac{K}{Z(K)}$ and $\theta: L' \longrightarrow K'$ be Lie superalgebra homomorphisms such that the following diagram is commutative,
\begin{center}
  \begin{tikzpicture}[>=latex]
\node (x) at (0,0) {\(L/Z(L)\times L/Z(L) \)};
\node (z) at (0,-3) {\(K/Z(K)\times K/Z(K)\)};
\node (y) at (3,0) {\(L'\)};
\node (w) at (3,-3) {\(K'\)};
\draw[->] (x) -- (y) node[midway,above] {$\mu$};
\draw[->] (x) -- (z) node[midway,left] {$\varphi ^{2} $};
\draw[->] (z) -- (w) node[midway,below] {$\rho$};
\draw[->] (y) -- (w) node[midway,right] {$\theta$};
\end{tikzpicture}\\
 \end{center} 
where $\mu( (\bar{l}, \bar{m})):=[l, m]$ for $ l, m \in L $ and similarly for $\rho( (\bar{r}, \bar{s})):= [r,s]$ for $r, s \in K$. Or, equivalently $\varphi$ and $\theta$ are defined in such a way that they are compatible, i.e., $\theta([l, m]) = [k, r]$, where $l, k, m, r \in L$ in which $k \in \varphi(l +Z(L))$ and $ r \in \varphi(m +Z(L))$.
Then the pair $(\varphi, \theta)$ is called {\it homoclinism} and if they are both isomorphisms, then $(\varphi, \theta)$ is called {\it isoclinism}.
\end{definition}
 If $(\varphi, \theta)$ is an isoclinism between $L$ and $K$, then $L$ and $K$ are said to be isoclinic, which is denoted by $L\sim K$. Obviously isoclinism is an equivalence relation, and hence it produces a partition on the class of all Lie superalgebras into equivalence classes called isoclinism classes. 
 
Here we listed some results on isoclinism of Lie superalgebras which are useful for us here and for details one can see \cite{Nayak2018}.
\begin{lemma}\label{Lem1}
If $L$ is a Lie superalgebra and $A$ is a abelian Lie superalgebra, then $L\sim L \oplus A$.
\end{lemma}
\begin{lemma}\label{Lem2}
Let $L$ be a Lie superalgebra and $I$ be a graded ideal. Then 
 $L/I \sim L/(I \cap L')$. In particular, if $I \cap L' = 0$ then $L \sim L/I$. Conversely, if $L'$ is finite dimensional and $L \sim L/I$, then $ I \cap L'=0.$
\end{lemma}
\begin{lemma}\label{Lem2'}
Let $L$ and $M$ be Lie superalgebras and $f: L \longrightarrow M$ be an onto homomorphism such that $\mbox{Ker}(f) \cap L'=0$, then $L \sim M$.
\end{lemma}
A Lie superalgebra $L$ is called a stem Lie superalgebra, whenever $Z(L) \subseteq L'$. The following properties of stem Lie superalgebras can be easily seen \cite{Nayak2018}.
\begin{lemma}\label{Lem3}
Suppose  $\mathcal{C}$ is an isoclinic family of Lie superalgebras. Then
\begin{enumerate}
\item[$(1)$] $\mathcal{C}$ contains a stem Lie superalgebra.
\item[$(2)$] Each finite dimensional Lie superalgebra $T \in \mathcal{C}$ is stem if and only if $T$ has minimal dimension in $\mathcal{C}$.
\end{enumerate}
\end{lemma}

The covers and multipliers for Lie algebras are defined and studied by Betten et al. It has been shown that unlike the case of groups the covers of a finite dimensional Lie algebras are isomorphic \cite{Batten1993, BS1996}. 
Likewise covers and multipliers for Lie superalgebras has recently been defined and studied \cite{SNayak2018}. 
\par
An extension of a Lie superalgebra $L$ is a short exact sequence 
\begin{equation}\label{eq0}
0 \longrightarrow  M \longrightarrow K \longrightarrow L \longrightarrow 0. 
\end{equation}
Since $e: M \longrightarrow e(M) = \ker(f)$ is an isomorphism we identify $M$ and $e(M)$. An extension of $L$ is then same as an epimorphism $f: K \longrightarrow L$. If $L$ is a Lie superalgebra generated by a $\mathbb{Z}_{2}$-graded set $X = X_{\bar{0}} \cup X_{\bar{1}}$  and $\phi : X \rightarrow L$ is a degree zero map,
then there exists a free Lie superalgebra $F$ and  $\psi: F \rightarrow L$ extending $\phi$. Let $R = \mbox{Ker} (\psi)$. The extension 
\begin{equation}\label{eq1}
0 \longrightarrow R \longrightarrow F \longrightarrow L \longrightarrow 0
\end{equation} 
is called a {\it free presentation} of $L$ and is denoted by $(F, \psi)$. With this free presentation of $L$, we define {\it multiplier} of $L$ as  
$$ 
\mathcal{M}(L) = \frac{[F,F]\cap R}{[F, R]}.
$$

\par

A homomorphism from an extension $f: K \longrightarrow L$ to another extension $f': K' \longrightarrow L$ is a Lie superalgebra homomorphism $g: K \longrightarrow K'$ satisfying $f = f' \circ g$; in other words, we have the following commutative diagram, 
\begin{center}
\begin{tikzpicture}[>=latex]
\node (K) at (0,0) {\(K\)};
\node (K^{'}) at (0,-2) {\(K^{'}\)};
\node (L) at (2.5,0) {\(L\)};
\draw[->] (K) -- (L) node[midway,above] {$f$};
\draw[->] (K) -- (K^{'}) node[midway,left] {$g$};
\draw[->,] (K^{'}) -- (L) node[midway,below] {$f^{'}$};
\end{tikzpicture}\\
\end{center}
A {\it central extension} of $L$ is an extension \eqref{eq0} such that $M \subseteq Z(K)$. The central extension is said to be a {\it stem extension} of $L$ if $M \subseteq Z(K)\cap K' $. Finally, we call the stem extension a {\it stem cover} if $M \cong \mathcal{M}(L)$ and in this case $K$ is said to be a cover of Lie superalgebra $L$. 
\begin{lemma} \label{Lem3'} \cite{Nayak2018}
Stem covers exist for each Lie superalgebras.
\end{lemma}
\begin{definition}
Let \eqref{eq0} be a stem extension of $L$.
The stem extension is {\it maximal} if every homomorphism of any other stem extension of $L$ onto $0 \longrightarrow M \longrightarrow  K \longrightarrow   L \longrightarrow 0$ is necessarily an isomorphism. 
\end{definition} 
The following is an important result which we use further.
\begin{lemma} \cite{Nayak2018} \label{Lem4'}
 The stem extension $$0 \longrightarrow M \longrightarrow  K \longrightarrow   L \longrightarrow 0$$ of finite dimensional Lie superalgebra $L$ is maximal (or, equivalently dimension of $K$ is maximal) if and only if it is a stem cover of $L$.
\end{lemma}
Finally below is another useful Lemma.
\begin{lemma}\label{Lem4}
Let the pair $(\varphi,\theta)$ be isoclinism of Lie superalgebras $L$ and $M$. Then
\begin{enumerate}
\item $\varphi(x+Z(L))=\theta(x)+Z(M) $;
\item $\theta([x,y])=[\theta(x),z]$, for all $x \in L',~y \in L,~z+Z(M)=\varphi(y+Z(L))$.
\end{enumerate}
 \end{lemma}
\begin{proof}

We have the following commutative diagram,

\begin{center}
  \begin{tikzpicture}[>=latex]
\node (x) at (0,0) {\(L/Z(L)\times L/Z(L) \)};
\node (z) at (0,-3) {\(M/Z(M)\times M/Z(M)\)};
\node (y) at (3,0) {\(L^{'}\)};
\node (w) at (3,-3) {\(M^{'}\)};
\draw[->] (x) -- (y) node[midway,above] {$\mu$};
\draw[->] (x) -- (z) node[midway,left] {$\varphi ^{2} $};
\draw[->] (z) -- (w) node[midway,below] {$\rho$};
\draw[->] (y) -- (w) node[midway,right] {$\theta$};
\end{tikzpicture}.
 \end{center}
 Let $x=[x_{1},x_{2}] \in L'$, then 
\begin{equation} 
\begin{split}
\theta(x)+Z(M)&= \theta[x_{1},x_{2}]+Z(M)\\
 &=\theta\mu(x_{1}+Z(L),x_{2}+Z(L))+Z(M)\\
 &=\rho \varphi^{2}(x_{1}+Z(L),x_{2}+Z(L))+Z(M)\\
  &=\rho \big(\varphi(x_{1}+Z(L)),\varphi(x_{2}+Z(L))\big)+Z(M)\\
  &= [\varphi(x_{1}+Z(L)),\varphi(x_{2}+Z(L))]+Z(M)\\
  &= \varphi[x_{1}+Z(L),x_{2}+Z(L)]=\varphi([x_{1},x_{2}]+Z(L))=\varphi(x+Z(L))
 \end{split}
\end{equation}
which proves $(1)$.
To prove the second part, consider for $y \in L$, we have $z+Z(M)=\varphi(y+Z(L))$.
\begin{equation} 
\begin{split}
\theta([x,y])=\theta\mu(x+Z(L),y+Z(L))
&= \rho \varphi ^{2}(x+Z(L),y+Z(L))\\
 &=\rho \big(\varphi(x+Z(L)),\varphi(y+Z(L))\big)\\
  &= [\varphi(x+Z(L)),\varphi(y+Z(L))]\\
  &=[\theta(x)+Z(M),z+Z(M)]= [\theta(x),z].
\end{split}
\end{equation}
\end{proof}

 In Section \ref{sec1} we recall some notations of Lie superalgebras.
In Section \ref{sec2} we show that isoclinism and isomorphism are equivalent for finite dimensional Lie superalgebras with the same dimension. Specifically we introduced factor set in Lie superalgebras to do the same. This result was first done for Lie algebras by Moneyhun \cite{Moneyhun1994} in 1994 and for $n$-Lie algebras by Eshrati et al. \cite{EMS2016}.  Moreover as an application to concept of isoclinism it is shown that covers for finite dimensional Lie superalgebras are isomorphic in Section \ref{sec3}. The same result hold for Lie algebras case see \cite{Moneyhun1994, MoneyhunD1994, Batten1993, BS1996} and for $n$-Liealgebras see \cite{EMS2016}.
\section{Preliminaries}\label{sec1}
Here we fix some terminology on Lie superalgebras (see \cite{KAC1977, Musson2012}) and recall some notions. Set $\mathbb{Z}_{2}=\{\bar{0}, \bar{1}\}$ is a field. A $\mathbb{Z}_{2}$-graded vector space $V$ is simply a direct sum of vector spaces $V_{\bar{0}}$ and $V_{\bar{1}}$, i.e., $V = V_{\bar{0}} \oplus V_{\bar{1}}$. It is also referred as a superspace. We consider all vector superspaces and superalgebras are over $\mathbb{F}$ (characteristic of $\mathbb{F} \neq 2,3$). Elements in $V_{\bar{0}}$ (resp. $V_{\bar{1}}$) are called even (resp. odd) elements. Non-zero elements of $V_{\bar{0}} \cup V_{\bar{1}}$ are called homogeneous elements. For a homogeneous element $v \in V_{\sigma}$, with $\sigma \in \mathbb{Z}_{2}$ we set $|v| = \sigma$ is the degree of $v$. A vector subspace $U$ of $V$ is called $\mathbb{Z}_2$-graded vector subspace(or superspace) if $U= (V_{\bar{0}} \cap U) \oplus (V_{\bar{1}} \cap U)$. We adopt the convention that whenever the degree function appeared in a formula, the corresponding elements are supposed to be homogeneous. 
\smallskip

A {\it Lie superalgebra} is a superspace $L = L_{\bar{0}} \oplus L_{\bar{1}}$ with a bilinear mapping
$ [., .] : L \times L \rightarrow L$ satisfying the following identities:

\begin{enumerate}
\item $[L_{\alpha}, L_{\beta}] \subset L_{\alpha+\beta}$, for $\alpha, \beta \in \mathbb{Z}_{2}$ ($\mathbb{Z}_{2}$-grading),
\item $[x, y] = -(-1)^{|x||y|} [y, x]$ (graded skew-symmetry),
\item $(-1)^{|x||z|} [x,[y, z]] + (-1)^{ |y| |x|} [y, [z, x]] + (-1)^{|z| |y|} [z,[ x, y]] = 0$ (graded Jacobi identity),
\end{enumerate}
for all $x, y, z \in L$. 

\smallskip
For a Lie superalgebra $L = L_{\bar{0}} \oplus L_{\bar{1}}$, the even part $L_{\bar{0}}$ is a Lie algebra and $L_{\bar{1}}$ is a $L_{\bar{0}}$-module. If $L_{\bar{1}} = 0$, then $L$ is just Lie algebra. But in general a Lie superalgebra is not a Lie algebra. Lie superalgebra without even part, i.e., $L_{\bar{0}} = 0$, is an abelian Lie superalgebra, as $[x, y] = 0$ for all $x, y \in L$. A sub superalgebra of $L$ is a $\mathbb{Z}_{2}$-graded vector subspace which is closed under bracket operation. A $\mathbb{Z}_{2}$-graded subspace $I$ is a {\it graded ideal} of $L$ if $[I,L]\subseteq I$. The ideal $Z(L) = \{z\in L : [z, x] = 0\;\mbox{for all}\;x\in L\}$ is a graded ideal and it is called the {\it center} of $L$. Clearly, if $I$ and $J$ are graded ideals of $L$, then so is $[I, J]$. If $I$ is an ideal of $L$, the quotient Lie superalgebra $L/I$ inherits a canonical Lie superalgebra structure such that the natural projection map becomes a homomorphism.

\smallskip

By a {\it homomorphism} between superspaces $f: V \rightarrow W $ of degree $|f|\in \mathbb{Z}_{2}$, we mean a linear map satisfying $f(V_{\alpha})\subseteq W_{\alpha+|f|}$ for $\alpha \in \mathbb{Z}_{2}$. In particular, if $|f| = \bar{0}$, then the homomorphism $f$ is called homogeneous linear map of even degree. A Lie superalgebra homomorphism $f: L \rightarrow M$ is a  homogeneous linear map of even degree such that $f[x,y] = [f(x), f(y)]$ holds for all $x, y \in L$. The notions of {\it epimorphisms, isomorphisms} and {\it auotomorphisms} have the obvious meaning. Throughout for superdimension of Lie superalgebra $L$ we simply write $\dim(L)=(m\mid n)$, where $\dim L_{\bar{0}} = m$ and $\dim L_{\bar{1}} = n$.
\section{Factor set in Lie Superalgebras}\label{sec2}
The notation for factor set in Lie algebras was defined by Moneyhun in \cite{Moneyhun1994}. Here we start with defining factor set in Lie superalgebras and study some of its properties. 

\begin{definition}
Let $L=L_{\overline{0}}\oplus L_{\overline{1}}$ be a finite dimensional Lie superalgebra over a field $\mathbb{F}$. The bilinear map;
\[r:L/Z(L)\times L/Z(L) \longrightarrow Z(L)\]
is said to be a $factor~set$ if the following conditions are satisfied:
\begin{enumerate}
\item $r(\overline{a},\overline{b})\subseteq Z(L)_{\alpha+\beta}$, $\alpha, \beta \in \mathbb{Z}_{2}$
\item $r(\overline{a},\overline{b})=-(-1)^{deg(\overline{a})deg(\overline{b})}r(\overline{b},\overline{a})$
\item $r([\overline{a},\overline{b}],\overline{c})=r(\overline{a},[\overline{b},\overline{c}])-(-1)^{deg(\overline{a})deg(\overline{b})}r(\overline{b},[\overline{a},\overline{c}])$
\end{enumerate}
for all $\overline{a}, ~ \overline{b},  ~\overline{c} \in L/Z(L)$.
\end{definition}
\begin{lemma} \label{Lem5}
Let $L$ be a Lie superalgebra and $r$ be a factor set on $L$.
\begin{enumerate}
\item The set 
\[R=(Z(L),L/Z(L),r)=\{(x,\overline{a}):~x \in Z(L), \overline{a} \in L/Z(L)\} \]
is a Lie superalgebra under the component-wise addition and the multiplication
\[[(x_{1},\overline{a}),(x_{2},\overline{b})]:=(r(\overline{a},\overline{b}),[\overline{a},\overline{b}]).\]
\item $Z_{R}=\{(x,0) \in R:~x \in Z(L)\} \cong Z(L).$
\end{enumerate}
\end{lemma}
\begin{proof}

Here we have $Z(L)$ is a graded ideal and $L/Z(L)$ is a quotient Lie superalgebra. Hence there is a natural $\mathbb{Z}_{2}$-grading on set $R$. Any $(x,\overline{a}) \in R $, is an even element when $x \in Z(L)_{\bar{0}},~~\overline{a} \in (L/Z(L))_{\overline{0}}$ and it is an odd element when $x \in Z(L)_{\bar{1}}, ~~\overline{a} \in (L/Z(L))_{\overline{1}}$. Now it is easy to check that $R$ is a vector superspace with comonent wise addition and furthermore with the multiplication defined in (1) it is a superalgebra. Now we intend to show that $R$ is a Lie superalgebra. 

The way $[ , ]$ is defined clearly it is a bilinear map.
Let $(x_{1},\overline{a}),(x_{2},\overline{b})$ be two homogeneous elements in $R$, then, $deg(a)=deg((x_{1},\overline{a}))$ and $~deg(b)=deg((x_{2},\overline{b}))$. 
\begin{enumerate}
\item To check graded skew-symmetric property consider
\begin{equation}
\begin{split}
 [(x_{1},\overline{a}),(x_{2},\overline{b})]=(r(\overline{a},\overline{b}),[\overline{a},\overline{b}]) & = -(-1)^{deg(\overline{a})deg(\overline{b})}(r(\overline{b},\overline{a}),[\overline{b},\overline{a}]) \\
 & = -(-1)^{deg(\overline{a})deg(\overline{b})}[(x_{2},\overline{b}),(x_{1},\overline{a})] \\
 &= -(-1)^{deg((x_{1},\overline{a}))deg((x_{2},\overline{b}))}[(x_{2},\overline{b}),(x_{1},\overline{a})].\\
\end{split}
\end{equation}
\item To check the graded Jacobi identity, consider
\begin{equation} 
\begin{split}
[[(x_{1},\overline{a}),(x_{2},\overline{b})],(x_{3},\overline{c})] & = [(r(\overline{a},\overline{b}),[\overline{a},\overline{b}]),(x_{3},\overline{c})] \\
 & = (r([\overline{a},\overline{b}],\overline{c}),[[\overline{a},\overline{b}],\overline{c}] ) \\
   & = \bigg( \big(r(\overline{a},[\overline{b},\overline{c}])-(-1)^{deg(\overline{a})deg(\overline{b})}r(\overline{b},[\overline{a},\overline{c}]) \big), \big( [\overline{a},[\overline{b},\overline{c}]]-(-1)^{deg(\overline{a})deg(\overline{b})}[\overline{b},[\overline{a},\overline{c}]] \big)   \bigg) \\
    & = \bigg( \big(r(\overline{a},[\overline{b},\overline{c}]),[\overline{a},[\overline{b},\overline{c}]] \big) -(-1)^{deg(\overline{a})deg(\overline{b})} \big(r(\overline{b},[\overline{a},\overline{c}]),[\overline{b},[\overline{a},\overline{c}]] \big) \bigg) \\
    &=[(x_{1},\overline{a}),\big( r(\overline{b},\overline{c}),[\overline{b},\overline{c}]\big)]-(-1)^{deg(\overline{a})deg(\overline{b})}[(x_{2},\overline{b}),\big( r(\overline{a},\overline{c}),[\overline{a},\overline{c}]\big)] \\
    &= [(x_{1},\overline{a}),[(x_{2},\overline{b}),(x_{3},\overline{c})]]-(-1)^{deg(\overline{a})deg(\overline{b})}[(x_{2},\overline{b}),[(x_{1},\overline{a}),(x_{3},\overline{c})]].\\
\end{split}
\end{equation}

Hence, $R$ is a Lie superalgebra with the given multiplication. Also proof of part(ii) is obvious.
\end{enumerate}
\end{proof}
Below is the Lemma which states that factor sets exist for any Lie superalgebras.
\begin{lemma}\label{Lem6}
There exists a factor set $r$, for any Lie superalgebra $L$ such that $L\cong (Z(L),L/Z(L),r).$
\end{lemma}
\begin{proof}
Let $K$ be the vector superspace and let it be the complement of $Z(L)$ in $L$, i.e., $L=K \oplus Z(L)$. Let us define $T:L/Z(L)\rightarrow L$ by $T(\overline{a})=T(a+Z(L))=T(k+l+Z(L))=k$, for all $\overline{a} \in L/Z(L)$ with $a=k+l$ where $k \in K$ and $l \in Z(L)$. Here, $T$ is a well-defined map and is also homogeneous linear map of even degree. Clearly we have $\overline{T(\overline{a})}=\overline{a}$. Now
for $\overline{a}, \overline{b} \in L/Z(L)$, consider $[\overline{a},\overline{b}]=[k+l+Z(L),k'+l'+Z(L)]=[k,k']+Z(L)$, then
\begin{align*} 
 [T(\overline{a}),T(\overline{b})]-T[\overline{a},\overline{b}] +Z(L)&= [k,k']-T(\overline{[k,k']})+Z(L) \\
 &=[k, k']- \overline{T(\overline{[k, k']})}= 0+Z(L)
\end{align*}
implies $[T(\overline{a}),T(\overline{b})]-T[\overline{a},\overline{b}]  \in Z(L)$.
Define
\[r:L/Z(L)\times L/Z(L) \longrightarrow Z(L)\] by, $~~~~~~~~~~~~~~~~~~~~~~~~~~~~~r(\overline{a},\overline{b})=[T(\overline{a}),T(\overline{b})]-T[\overline{a},\overline{b}]$. We show that $r$ is a factor set.

Since $T$ is a homogeneous linea map of even degree and bracket respects gradation, so $r(\overline{a}, \overline{b}) \in (L/Z(L))_{\alpha+\beta}$ povided $\overline{a} \in L/Z(L))_{\alpha}$ and $\overline{b} \in L/Z(L))_{\beta}$. To check skew-symmetric property of $r$, take 
\begin{equation*}
\begin{split}
r(\overline{a}, \overline{b}) &= [T(\overline{a}), T(\overline{b})-T([\overline{a}, \overline{b}])\\
&= - (-1)^{deg(\overline{a}) deg(\overline{b}) }([T(\overline{b}), T(\overline{a})-T([\overline{b}, \overline{a}])\\
&= -(-1)^{deg(\overline{a}) deg(\overline{b})} r(\overline{b}, \overline{a})
\end{split}
\end{equation*}
for all $\overline{a}, \overline{b} \in L/Z(L)$.
 To check Jacobi identity, take
 \begin{equation*} 
\begin{split}
r([\overline{a},\overline{b}],\overline{c}) &= [T([\overline{a},\overline{b}]),T(\overline{c})]-T[[\overline{a},\overline{b}],\overline{c}]\\
&= [[T(\overline{a}),T(\overline{b})]+Z(L),T(\overline{c})+Z(L)]-T[[\overline{a},\overline{b}],\overline{c}]\\
&= [[T(\overline{a}),T(\overline{b})],T(\overline{c})]+Z(L)-T[[\overline{a},\overline{b}],\overline{c}]\\
&= \big([T(\overline{a}),[T(\overline{b}),T(\overline{c})]]-(-1)^{deg(\overline{a})deg(\overline{b})}[T(\overline{b}),[T(\overline{a}),T(\overline{c})]]\big)+Z(L)-T[\overline{a},[\overline{b},\overline{c}]]\\
&~~~~~~~~~~~~~~~~~~~~~~~~~~~~~~~~~~~~~~~~~+(-1)^{deg(\overline{a})deg(\overline{b})}T[\overline{b},[\overline{a},\overline{c}]]\\
&= \big([T(\overline{a}),[T(\overline{b}),T(\overline{c})]]+Z(L)-T[\overline{a},[\overline{b},\overline{c}]]\big)-(-1)^{deg(\overline{a})deg(\overline{b})} \big([T(\overline{b}),[T(\overline{a}),T(\overline{c})]]+Z(L)-T[\overline{b},[\overline{a},\overline{c}]]\big)\\
&=r(\overline{a},[\overline{b},\overline{c}])-(-1)^{deg(\overline{a})deg(\overline{b})}r(\overline{b},[\overline{a},\overline{c}])
\end{split}
\end{equation*}
for all $\overline{a}, \overline{b}, \overline{c} \in L/Z(L)$.
Consider, $\theta:(Z(L),L/Z(L),r)\rightarrow L$ is given by, $\theta(x,\overline{a})= x+T(\overline{a})$, for  $x \in  Z(L)$ and $\overline{a}=a+Z(L)=k+Z(L)$ where $k \in K$. Suppose,  
$(x,\overline{a})=(y,\overline{b})\Leftrightarrow x+ k = y+ k^{\prime}\Leftrightarrow x-y =  k-k^{\prime} \in Z(L)\cap K={0}\Leftrightarrow \theta(x,\overline{a})=\theta(y,\overline{b})$. Hence, $\theta$ is well-defined map and it is injective. Evidently $\theta$ is surjective. Also $\theta$ is a homogeneous linear map of even degree and as
 \begin{equation} 
\begin{split}
\theta([(x,\overline{a}),(y,\overline{b})]) &=\theta(r(\overline{a},\overline{b}),[\overline{a},\overline{b}])\\
&= r(\overline{a},\overline{b})+T([\overline{a},\overline{b}])
= [T(\overline{a}),T(\overline{b})]  \\
&= [x+T(\overline{a}),y+T(\overline{b})]
=[\theta(x,\overline{a}),\theta(y,\overline{b})],
\end{split}
\end{equation}
so it is a Lie superalgebra homomorphism, hence $\theta$ is the required isomorphism.
\end{proof}
Following is the Lemma which gives a connection between two stem Lie superalgebras.
 \begin{lemma}\label{Lem7}
Let $L$ be a stem Lie superalgebra in an isoclinism family of Lie superalgebras $\mathcal{C}$. Then for any stem Lie superalgebra $M$ of $\mathcal{C}$, there exists a factor set $r$ over $L$ such that $M\cong (Z(L),L/Z(L),r)$. 
\end{lemma}
\begin{proof}
Let $L\sim M$ and let the pair $(\varphi,\theta)$ be isoclinism of Lie superalgebras $L$ and $M$. We have $L' \subseteq Z(L)$ and $M' \subseteq Z(M)$. Then, by Lemma \ref{Lem4}, $\theta(Z(L))=Z(M)$. According to Lemma \ref{Lem6}, there exists a factor set $s$ such that \[M\cong (Z(M),M/Z(M),s).\] Now, define the factor set 
\[ r:L/Z(L)\times L/Z(L) \longrightarrow Z(L)\]
by, \[~~~~~~~~r(\overline{a},\overline{b})=\theta^{-1}(s(\varphi(\overline{a}),\varphi(\overline{b}))),\] where $\overline{a},~\overline{b} \in L/Z(L)$. 
Further define 
\[\psi:(Z(L),L/Z(L),r)\longrightarrow (Z(M),M/Z(M),s)\]
by, \[~~~~~~~~~~\psi(x,\overline{a})=(\theta(x),\varphi(\overline{a})).\]
Since, $\theta$  and $\varphi$ are isomorphisms, so $\psi$ is a bijection and is also homogeneous linear map of even degree. Finally, consider  

\begin{equation} 
\begin{split}
\psi([(x,\overline{a}),(y,\overline{b})]) &=\psi(r(\overline{a},\overline{b}),[\overline{a},\overline{b}])\\
&= \big(\theta(r(\overline{a},\overline{b})),\varphi([\overline{a},\overline{b}])\big)
=  \big(s(\varphi(\overline{a}),\varphi(\overline{b})),\varphi([\overline{a},\overline{b}])\big)\\
&= [(\theta(x),\varphi(\overline{a})),(\theta(y),\varphi(\overline{b}))]
=[\psi(x,\overline{a}),\psi(y,\overline{b})],
\end{split}
\end{equation}
implies $\psi$ is a Lie superalgebra homomorphism.
Hence, $\psi$ is an isomorphism, i.e., \[ (Z(L),L/Z(L),r)\cong (Z(M),M/Z(M),s).\] It follows that, \[M\cong (Z(L),L/Z(L),r).\] 
\end{proof}

\begin{lemma}\label{Lem8}
Let $L$ be a Lie superalgebra, $r$ and $s$ be two factor sets over $L$. Assume that
\[R=(Z(L),L/Z(L),r),~~~~~~~~Z_{R}=\{(x,0) \in R:~x \in Z(L)\} \cong Z(L),\]
\[S=(Z(L),L/Z(L),s),~~~~~~~~Z_{S}=\{(x,0) \in R:~x \in Z(L)\} \cong Z(L).\]
Let $\lambda$ is an isomorphism from $R$ to $S$ satisfying $\lambda(Z_{R})=Z_{S}$, then the restriction of $\lambda$ on $L/Z(L)$ and $Z(L)$ define the automorphisms $\mu \in Aut(L/Z(L))$ and $\nu \in Aut(Z(L))$, respectively.
 \end{lemma}
\begin{proof}
By assumption, $\lambda(Z_{R})=Z_{S}$, then define the map $\overline{\lambda}:R/Z_{R}\longrightarrow S/Z_{s}$ given by, $\overline{\lambda}((x,\overline{a})+Z_{R})=\lambda(x,\overline{a})+Z_{S}$. Here $\overline{\lambda}$ is an isomorphism. Define 
$\mu$ such that the following diagram is commutative;
\begin{center}
  \begin{tikzpicture}[>=latex]
\node (x) at (0,0) {\(L/Z(L) \)};
\node (z) at (0,-3) {\(R/Z_{R}\)};
\node (y) at (3,0) {\(L/Z(L)\)};
\node (w) at (3,-3) {\(S/Z_{S}\)};
\draw[->] (x) -- (y) node[midway,above] {$\mu$};
\draw[->] (x) -- (z) node[midway,left] {$\phi$};
\draw[->] (z) -- (w) node[midway,below] {$\overline{\lambda}$};
\draw[->] (y) -- (w) node[midway,right] {$\psi$};
\end{tikzpicture}\\
\end{center}
 where $\phi$ is defined as $\phi( \overline{a})= (0, \overline{a})+Z_{R}$ and $\psi(\overline{a})=(0, \overline{a})+Z_{S}$. In other words $\lambda(0,\overline{a})+Z_{S}=(0,\mu(\overline{a}))+Z_{S}$, for all $\overline{a} \in L/Z(L)$. Certainly $\phi, \psi$ are both homogeneous linear maps of even degree. Moreover for all $\overline{a},\overline{b} \in L/Z(L)$, we have
 \begin{align*} 
 \phi[\overline{a}, \overline{b}]&=(0,[\overline{a}, \overline{b}])+Z_{R}\\
 &= (r(\overline{a}, \overline{b}), [\overline{a}, \overline{b}])
 =[(0,\overline{a}), (0, \overline{b})]\\
 &= [(0, \overline{a})+Z_{R}, (0, \overline{b})+Z_{R}]
 =[\phi(\overline{a}), \phi(\overline{b})].
 \end{align*} So, $\phi$ is a homomorphism and similarly $\psi$ is also a homomorphism. Since, $\lambda$ is an isomorphism, $\mu$ is a homogeneous linear map of even degree. Now, to check the map $\mu$ is well-defined and injective, let $\overline{a} = \overline{b}$ implies $(0,\overline{a})= (0,\overline{b})$ in $L/Z(L)$, so in $R/Z_{R}$ we have $(0,\overline{a})+Z_{R} = (0,\overline{b})+Z_{R}$, 
\begin{align*} 
\Leftrightarrow \overline{\lambda}((0,\overline{a})+Z_{R}) &= \overline{\lambda}((0,\overline{b})+Z_{R})\\
\Leftrightarrow \lambda(0,\overline{a})+Z_{S} &= \lambda(0,\overline{b})+Z_{S}\\
\Leftrightarrow (0,\mu(\overline{a}))+Z_{S} &= (0,\mu(\overline{b}))+Z_{S}\\
\Leftrightarrow (0,\mu(\overline{a})-\mu(\overline{b})) &\in Z_{S}.
\end{align*}
But any element in $Z_{S}$ is of the form $(x, 0)$, for $x \in Z(L)$, which implies $\Leftrightarrow \mu(\overline{a})-\mu(\overline{b}))=0 ~in ~L/Z(L)
\Leftrightarrow \mu(\overline{a})=\mu(\overline{b})) ~in ~L/Z(L)$.

For each $\overline{b} \in L/Z(L)$, there exist $\overline{a} \in L/Z(L)$ such that $\overline{\lambda}((0,\overline{a})+Z_{R})=(0,\overline{b})+Z_{S}$ and on the other hand, 
$\overline{\lambda}((0,\overline{a})+Z_{R})=\lambda(0,\overline{a})+Z_{S}=(0,\mu(\overline{a}))+Z_{S}$. Hence, $\mu(\overline{a})=\overline{b}$, i.e., $\mu$ is surjective. Now, for
\begin{equation} 
\begin{split}
(0,\mu([\overline{a}, \overline{b}]))+Z_{S}&=\lambda(0,[\overline{a},\overline{b}])+Z_{S}\\
&=\overline{\lambda}(0,[\overline{a},\overline{b}]+Z_{R})
=\overline{\lambda}\phi([\overline{a},\overline{b}])\\
&=[\overline{\lambda}\phi(\overline{a}),\overline{\lambda}\phi(\overline{b})]
=[\psi\mu(\overline{a}),\psi\mu(\overline{b})]\\
&=\psi[\mu(\overline{a}),\mu(\overline{b})]
=(0,[\mu(\overline{a}),\mu(\overline{b})]+Z_{S},
\end{split}
\end{equation}
i.e, $\mu([\overline{a},\overline{b}]=[\mu(\overline{a}),\mu(\overline{b})]$ and thus $\mu$ is an automorphism. Consider the map $\tilde{\lambda}: Z_{R} \longrightarrow Z_{S}$ is defined as $\tilde{\lambda}(x,0)=\lambda(x,0)$ for all $x \in Z(L)$, is an isomorphism. Define $\nu$ in such a way that the following diagram is commutative;
\begin{center}
\begin{tikzpicture}[>=latex]
\node (x) at (0,0) {\( Z(L) \)};
\node (z) at (0,-3) {\(Z_{R}\)};
\node (y) at (3,0) {\(Z(L)\)};
\node (w) at (3,-3) {\(Z_{S}\)};
\draw[->] (x) -- (y) node[midway,above] {$\nu$};
\draw[->] (x) -- (z) node[midway,left] {$\phi$};
\draw[->] (z) -- (w) node[midway,below] {$\tilde{\lambda}$};
\draw[->] (y) -- (w) node[midway,right] {$\psi$};
\end{tikzpicture}\\
\end{center}
 
i.e. $\lambda(x,0)=(\nu(x),0)$, for all $x \in Z(L)$. It is easy to check that $\nu$ is an automorphism.
\end{proof}
\begin{lemma}\label{lem9}
Let $L$ be a Lie superalgebra and $R, S, Z_{R}$ and $Z_{S}$ be as in the previous lemma.
\begin{enumerate}
\item Consider $\lambda : R \longrightarrow S$ is a Lie superalgebra isomorphism such that $\lambda(Z_{R})= Z_{S}$. Let $\mu \in Aut(L/Z(L))$ and $\nu \in Aut(Z(L))$ be the automorphisms induced by $\lambda$. Then there exists a homogeneous linear map of even degree, $\gamma: L/Z(L) \longrightarrow Z(L)$ such that,\[\nu(r(\overline{a}, \overline{b})+ \gamma[\overline{a}, \overline{b}])=s(\mu(\overline{a}), \mu(\overline{b})).\]
\item If $\mu \in Aut(L/Z(L))$, $\nu \in Aut(Z(L))$ and $\delta: L/Z(L) \longrightarrow Z(L)$ is a homogeneous linear map of even degree such that $$\nu(r(\overline{a}, \overline{b})+ \delta[\overline{a}, \overline{b}])=s(\mu(\overline{a}), \mu(\overline{b}))$$ holds, then there exists an isomorphism $\lambda:R \longrightarrow S$ which is induced by $\mu$ and $\nu$ satisfying $\lambda(Z_{R})=Z_{S}$.
\end{enumerate}
\end{lemma}
\begin{proof}
We have,  $\lambda(0, \overline{a})+Z_{S}=(0, \mu(\overline{a}))+Z_{S}$ which means $\lambda(0, \overline{a})-(0, \mu(\overline{a})) \in Z_{S}$. Say,  $\lambda(0, \overline{a})-(0, \mu(\overline{a})=(x_{\overline{a}}, 0)$, for some $x_{\overline{a}} \in Z(L)$. Define the map $\gamma: L/Z(L) \longrightarrow Z(L)$ by $\gamma(\overline{a})=x_{\overline{a}}$, for all $\overline{a} \in L/Z(L)$. It is easy to see that the map $\gamma$ is well defined. For $\overline{a}, \overline{b} \in \frac{L_{\tau}+Z(L)}{Z(L)}$ with $\tau \in \mathbb{Z}_{2}$,  we have 
\begin{align*}
(\gamma(\overline{a}+\overline{b}), 0) &=(x_{\overline{a+b}},0)\\
&=\lambda(0, \overline{a+b})- (0, \mu(\overline{a+b}))\\
&=\lambda(0, \overline{a})-(0, \mu(\overline{a}))+\lambda(0, \overline{b})-(0, \mu(\overline{b}))\\
&=(x_{\overline{a}}, 0)+(y_{\overline{b}},0)
= (\gamma(\overline{a})+\gamma(\overline{b}), 0).
\end{align*}
Hence, we get $\gamma(\overline{a}+\overline{b})=\gamma(\overline{a})+\gamma(\overline{b})$ implies $\gamma$ is a linear map. Further as  $\lambda$ and $\mu$ are Lie superalgebra isomorphisms, that immediately implies $\gamma$ is a homogeneous linear map of even degree. We have for $x \in Z(L)$, $\lambda(x, 0)=(\nu(x),0)$. 
\begin{align*}
\lambda(x, \overline{a}) &=(\nu(x), 0)+(0, \mu(\overline{a}))+(\gamma(\overline{a}),0)\\
&= (\nu(x)+\gamma(\overline{a}), \mu(\overline{a})).
\end{align*}
On one hand, \begin{align*}
\lambda[(x, \overline{a}), (y, \overline{b})]&=[\lambda(0, \overline{a}), \lambda(0, \overline{b})]\\
&= [(0, \mu(\overline{a}))+(x_{\overline{a}}, 0), (0, \mu(\overline{b}))+(y_{\overline{b}}, 0)]\\
&=[(\gamma(\overline{a}), \mu(\overline{a})), (\gamma(\overline{b}), \mu(\overline{b}))]
=\left( s(\mu(\overline{a}),\mu(\overline{b})), [\mu(\overline{a}), \mu(\overline{b})]\right).
\end{align*}
On the other hand, $\lambda[(x, \overline{a}), (y, \overline{b})]= \lambda \left( r(\overline{a}, \overline{b}), [\overline{a}, \overline{b}] \right)= \left( \nu(r(\overline{a}, \overline{b}))+ \gamma[\overline{a}, \overline{b}], [\mu(\overline{a}), \mu(\overline{b})]\right)$. Hence, 
$$\nu(r(\overline{a}, \overline{b})+ \gamma[\overline{a}, \overline{b}])=s(\mu(\overline{a}), \mu(\overline{b})),$$ which proves our first assertion.
\par
Let us define $\lambda$ by \begin{equation*}
\begin{split}
\lambda: R &\longrightarrow S \\
\lambda(x, \overline{a}) &= ( \nu(x)+ \delta(\overline{a}), \mu(\overline{a})),
\end{split}
\end{equation*}
 $\forall ~~ \overline{a} \in \frac{L_{\tau}+Z(L)}{Z(L)}$ and $\forall ~~x \in L_{\tau} \cap Z(L)$ with $\tau \in \mathbb{Z}_{2}$. Here $\lambda$ is a well defined map and also a bijection. Further as $\nu, \delta, \mu$ are all homogeneous linear maps of even degree hence so is $\lambda$. Now, 
\begin{align*}
\lambda[(x, \overline{a}), (x, \overline{b})] &= \lambda \left(r( \overline{a}, \overline{b}), [\overline{a}, \overline{b}]\right)\\
&= \left( \nu(r(\overline{a}, \overline{b})+[\overline{a}, \overline{b}]+ \delta[\overline{a}, \overline{b}], \mu[\overline{a}, \overline{b}] \right)\\
&= \left( s(\mu(\overline{a}), \mu(\overline{b})), \mu[\overline{a}, \overline{b}]\right).
\end{align*}
Also \begin{align*}
[\lambda(x, \overline{a}), \lambda(y, \overline{b})] 
&=[(\nu(x)+\delta(\overline{a}), \mu(\overline{a})), (\nu(y)+\delta(\overline{b}), \mu(\overline{b}))]\\
&= \left( s(\mu(\overline{a}), \mu(\overline{b})), [\mu(\overline{a}), \mu(\overline{b})] \right).
\end{align*}
Finally, we get $\lambda[(x, \overline{a}), (y, \overline{b})] =[\lambda(x, \overline{a}), \lambda(y, \overline{b})]$, so $\lambda$ is a Lie superalgebra isomorphism. Evidently $\lambda(x, 0)=(\nu(x), 0)$ and $\lambda(0, \overline{a})= \left( \delta(\overline{a}), \mu(\overline{a})\right)$. Now, 
\begin{align*}
\lambda(0, \overline{a})+Z_{S} &= \left( \nu(0)+\delta(\overline{a}), \mu(\overline{a})\right) + Z_{S}\\
&=(0, \mu(\overline{a}))+(\delta(\overline{a}), 0)+Z_{S}\\
&= (0, \mu(\overline{a}))+Z_{S},
\end{align*}
where the  equality holds if $(0, \mu(\overline{a})) \in S$ satisfying $\lambda(Z_{R})= Z_{S}$ which proves our second assertion.
\end{proof}
The following Theorem plays an important role in proving the main result of this section. This result was first proved by Moneyhun for Lie algebras \cite{Moneyhun1994}.
\begin{theorem}\label{Th1}
Let $L$ and $M$ be two finite dimensional stem Lie superalgebras. Then $L \sim M$ if and only if $L \cong M$.
\end{theorem}
\begin{proof}
If $L$ and $M$ are finite dimensional stem Lie superalgebras such that $L \cong M$,  then clearly $L \sim M$. Conversely, let $L \sim M$. By Lemma \ref{Lem6} we have, $L \cong (Z(L), \frac{L}{Z(L)}, r)=R$ and also using Lemma \ref{Lem7}, $M \cong (Z(L), \frac{L}{Z(L)}, s)=S$. Let $(\alpha, \beta)$ be the isoclinism between the Lie superalgebras $R$ and $S$. Certainly $Z(L) \cong Z_{R}$ and also $Z(R) \cong Z(L)$, hence $Z(R) \cong Z_{R}$. Similarly $Z(S) \cong Z_{S}$. As $Z_{R} \subseteq Z(R)$, we get $Z_{R}=Z(R)$. Let us consider the following commutative diagram;
\begin{center}
 \begin{tikzpicture}[>=latex]
\node (A_{1}) at (0,0) {\(\frac{L}{Z(L)}\times \cdots \times \frac{L}{Z(L)} \)};
\node (A_{2}) at (4,0) {\(\frac{R}{Z_R}\times \cdots \times \frac{R}{Z_R}\)};
\node (A_{3}) at (8,0) {\(R'\)};
\node (B_{1}) at (0,-2) {\(\frac{L}{Z(L)}\times \cdots \times \frac{L}{Z(L)}\)};
\node (B_{2}) at (4,-2) {\(\frac{S}{Z_S}\times \cdots \times \frac{S}{Z_S} \)};
\node (B_{3}) at (8,-2) {\(S'\)};

\draw[->] (A_{1}) -- (A_{2}) node[midway,above] {$\rho$};
\draw[->] (A_{2}) -- (A_{3}) node[midway,above] {$\theta$};
\draw[->] (B_{1}) -- (B_{2}) node[midway,below] {$\sigma$};
\draw[->] (B_{2}) -- (B_{3}) node[midway,below] {$\xi$};

\draw[->] (A_{1}) -- (B_{1}) node[midway,right] {$\alpha^n$};
\draw[->] (A_{2}) -- (B_{2}) node[midway,right] {$\omega^n$};
\draw[->] (A_{3}) -- (B_{3}) node[midway,right] {$\tau$};

\end{tikzpicture}
\end{center}

 in which
  \begin{equation*}
  \begin{split}
 \rho(\overline{a}, \overline{b}) = ((0, \overline{a})+Z_{R}, (0, \overline{b})+Z_{R}),\\
 \sigma(\overline{a}, \overline{b}) =  ((0, \overline{a})+Z_{S}, (0, \overline{b})+ Z_{S}),\\
\xi((x, \overline{a})+Z_{S}, (y,\overline{b})+Z_{S})= [(x,\overline{a}),(y, \overline{b})]=(s(\overline{a}, \overline{b}), [\overline{a}, \overline{b}]),\\
\theta((x, \overline{a})+Z_{S}, (y,\overline{b})+Z_{S})=(r(\overline{a}, \overline{b}), [\overline{a}, \overline{b}]).
 \end{split}
 \end{equation*}
 Let the map $\mu \in Aut(L/Z(L))$, be defined by $\alpha((0, \overline{a})+Z_{R})= (0, \mu(\overline{a}))+Z_{S}$, for all $\overline{a} \in L/Z(L)$. Again let $\nu \in Aut(Z(L))$ be defined by $\beta(x, 0)=(\nu(x),0)$, for all $x \in Z(L)$. Now for $\overline{a}, \overline{b} \in L/Z(L)$, consider
$$
\beta \theta((0, \overline{a})+Z_{R}, (0, \overline{b})+Z_{R}) = \beta[(0,\overline{a}),(0, \overline{b})]
$$
 and further \begin{align*}
 \xi \alpha((0, \overline{a})+Z_{R}, (0, \overline{b})+Z_{R})&=\xi((0, \mu(\overline{a}))+Z_{S}, (0, \mu(\overline{b}))+Z_{S} )\\
 &= [(0, \mu(\overline{a})), (0, \mu(\overline{b}))]\\
 &= (s((\mu(\overline{a}), \mu(\overline{b})), [\mu(\overline{a}), \mu(\overline{b})]).
 \end{align*}
 Hence, we have $\beta[(0,\overline{a}),(0, \overline{b})]
=s((\mu(\overline{a}), \mu(\overline{b})), [\mu(\overline{a}), \mu(\overline{b})])$.
 Let us define the map $\delta: L'/Z(L) \longrightarrow Z(L)$ as follows, 
 \begin{align*}
 \beta[(0, \overline{a}), (0, \overline{b})] &= \beta(r( \overline{a}, \overline{b}), [\overline{a}, \overline{b}])\\
 &=\beta(r( \overline{a}, \overline{b}), 0)+ \beta(0, [\overline{a}, \overline{b}])\\
 &=(\nu(r( \overline{a}, \overline{b})), 0)+ (\delta( [\overline{a}, \overline{b}]),t)\\
 &=\nu(r( \overline{a}, \overline{b})+\delta( [\overline{a}, \overline{b}]),t),\\
 \end{align*} 
 where $t \in L/Z(L)$ and hence, we get
 \[\nu(r( \overline{a}, \overline{b})+\delta( [\overline{a}, \overline{b}])=s((\mu(\overline{a}), \mu(\overline{b})).\]
 To apply Lemma 2.6, we may extend $\delta$ to $L/Z(L)$ by defining $0$ on the complement of $L^{'}/Z(L)$ in $L/Z(L)$. Then, we get $R \cong S$.
\end{proof}
\begin{theorem}\label{Th2}
Let $\mathcal{C}$ be an isoclinism family of finite dimensional Lie superalgebras. Then any $L \in \mathcal{C}$ can be expressed as $L= T \oplus A$ where $T$ is a stem Lie superalgebra and $A$ is some finite dimensional abelian Lie superalgebra.
\end{theorem}
\begin{proof}
By Lemma 1.4,  $\mathcal{C}$ contains a stem Lie superalgebra say $T$. Further $T \sim  T \oplus A $ for some abelian Lie superalgebra $A$ (Lemma \ref{Lem1}). Consider $L \in \mathcal{C}$ be any arbitrary Lie superalgebra and we have  $Z(L) \cap L'$ is a graded ideal of $L$. Let $M= L_{\overline{0}} \cap M \oplus L_{\overline{1}} \cap M $ be a complementary $\mathbb{Z}_{2}$-graded vector subspace to $Z(L) \cap L'$ in $Z(L)$. So, 
$$Z(L)= M \oplus Z(L) \cap L'$$ and clearly  $[M, L] \subseteq M$ implies $M$ is an ideal of $L$. Assume $L/M=T$ and as $Z(L)$ is direct sum of $Z(L)\cap L'$ and $M$ we have $M \cap L'=0$. By Lemma \ref{Lem2}, we get $L \sim L/M$. Now,
\begin{align*}
Z(T) = Z(\frac{L}{M})&= \frac{Z(L)}{M}\\
&\subseteq \frac{L'+M}{M}\cong \left(\frac{L}{M}\right)'=T'.
\end{align*}
Hence $T$ is a stem Lie superalgebra. As $M$ is an ideal of $L$ and $M \cap L' =0$, we get a $\mathbb{Z}_{2}$-graded vector subspace $K$ of $L$ containing $L'$  complementary to $M$. Now, $$[K, L] \subseteq [L, L]=L' \subseteq K$$ implies $K$ is an ideal. Further, $$ L \sim \frac{L}{M} \cong \frac{K \oplus M}{M} \cong K.$$ It is easy to check $K$ is a stem Lie superalgebra. Finally $T \sim L$ and $ L \sim K$ implies $T \sim K$. Using Theorem \ref{Th1}, $T \cong K$. Specifically $L= K \oplus M \cong T \oplus M$ as required.
\end{proof}
\begin{theorem}\label{Th3}
If $L$ and $K$ be two Lie superalgebra with same dimension. Then $L \sim K$ if and only if $L \cong K$.
\end{theorem}
\begin{proof}
Consider Lie superalgebras $L$ and $K$ with $L \sim K$. By Theorem \ref{Th2}, write $L= T \oplus M$ and $K= T_{1} \oplus M_{1}$ for stem Lie superalgebras $T, T_{1}$ and abelian Lie superalgebras $M, M_{1}$ respectively. As $T \sim T_{1}$  by Theorem \ref{Th1} $T \cong T_{1}$. Since, $L$ and $K$ are of same dimension, so $\dim M= \dim M_{1}$ and $M \cong M'$, which implies $T \oplus M \cong T_{1} \oplus M_{1}$. Therefore $L \cong K$. Converse is obvious.
\end{proof}
 Below is an example which shows that, two isoclinic Lie superalgebras of different dimension may not be isomorphic.
\begin{example}
Let $L=L_{\overline{0}}+L_{\overline{1}}$ be a $(2|1)$ dimensional Lie superalgebra with the basis $\{e_{1},e_{2},e_{3}\}$ and the commutator relations are defined by;
\[[e_{1},e_{2}]=e_{1},~~[e_{3},e_{3}]=e_{2},\]
and all other commutator relations are zero. Then $L'=<e_{1},e_{2}>$ and $Z(L)=0$ and hence, $L/Z(L) \cong L$.\\

Now, let $M=M_{\overline{0}}+M_{\overline{1}}$ be a $(3|1)$ dimensional Lie superalgebra with the basis $\{e_{1},e_{2},e_{3},e_{4}\}$ and  the commutator relations are defined by;
\[[e_{1},e_{2}]=e_{1},~~[e_{4},e_{4}]=e_{2},\]
and all other commutator relations are zero. Then $M'=<e_{1},e_{2}>$ and $Z(M)=\{e_{3}\}$ and hence, $M/Z(M)=\{ \overline{e_{1}},\overline{e_{2}},\overline{e_{4}}\}$, where $\overline{e_{i}}=e_{i}+Z(M)$, for $i=1,2,4$.\\

It is easy to verify that $L' \cong M'$ and $L/Z(L) \cong M/Z(M)$. Hence, one can deduce $L\sim M$ while $dim(L)\neq dim(M)$.
\end{example}
\section{Covers of finite dimensional Lie superalgebras}\label{sec3}

Here we show that for finite dimensional Lie superalgebras covers are isomorphic, using isoclinism. Atfirst we define the notation of universal element.
\begin{definition}
Let $L$ be a Lie superalgebra. Let $$ C(L)= \{(K, \lambda)| \lambda \in Hom(K, L), ~\lambda~ is ~onto ~and~ \mbox{Ker}(\lambda) \subseteq K' \cap Z(K)\},$$ for any Lie superalgebra $K$. The element $(T, \sigma) \in C(L)$ is called a universal element if there exists a map $\tau \in Hom(T, K)$ satisfying $\lambda \circ \tau = \sigma$ , ~ $\forall ~ (K, \lambda) \in C(L)$.
\end{definition}
\begin{lemma}\label{Lem10}
Let $K=K_{\overline{0}} \oplus K_{\overline{1}}$ be a Lie superalgebra of $\dim K=(m|n)$, then $Z(K) \cap K'$ is contained in any maximal Lie superalgebra of $K$.
\end{lemma}
\begin{proof}
Let $W=K_{\overline{0}} \cap W \oplus K_{\overline{1}} \cap W $ be any maximal subalgebra of $K$. There is an natural gradation for $Z(K) \cap K'+W$ is given by $Z(K) \cap K'+W = K_{\overline{0}} \cap (Z(K) \cap K'+ W)\oplus K_{\overline{1}} \cap (Z(K) \cap K'+ W)$. It is a subalgebra of $K$. Hence either $Z(K) \cap K'+W= W$ or $Z(K) \cap K'+W= K$.  If $Z(K) \cap K'+W= K$, then $K' = W'\subseteq W $, so $Z(K) \cap K' \subseteq W$ which is a contradiction, i.e., $Z(K) \cap K'+W= W$.
\end{proof}
\begin{lemma}\label{Lem11}
Let $L$ be a Lie superalgebra of dimension $(m|n)$. Let $(K, \lambda) \in C(L)$ and $M$ be any Lie superalgebra and $\sigma \in Hom(M, L)$ is onto. If
$\tau \in Hom(M, K)$ such that $\lambda \circ \tau = \sigma$, then $\tau$ is onto.
\end{lemma}
\begin{proof}
Consider $(K, \lambda) \in C(L)$ and for $l \in K$, as $\sigma$ is onto, $\sigma(m)=\lambda(l)$ for some $m \in M$. Now $\sigma(m)=\lambda (\tau(m))=\lambda(l)$ which implies $\tau(m)- l \in \mbox{Ker}(\lambda)$ and moreover $K= \mbox{Ker}(\lambda)+ \mbox{Img}(\tau)$. By assumption $\mbox{Ker}(\lambda) \subseteq Z(K) \cap K'$ and by Lemma \ref{Lem10}, $Z(K) \cap K'$ is contained in every maximal subalgebra of $K$. Suppose $\mbox{Img}(\tau) \neq K$. Now, $\mbox{Img}(\tau)$ is a subalgebra of $K$ and say $\mbox{Img}(\tau) \subset N$ for some maximal subalgebra $N \neq K$. But $$K \subseteq \mbox{Img}(\tau)+ (Z(K) \cap K') \subset N,$$ i.e., $K=N$ leads to a contradiction. Hence, $\mbox{Img}(\tau)=K$ as required.
\end{proof}
If $(K, \lambda) \in C(L)$ is universal element, then for each $(M,\tau)\in C(L)$, there exists homomorphism $\rho\in Hom(K, M) $ such that $\tau \circ \rho=\lambda$. By Lemma \ref{Lem11}, we have $\rho$ is onto, hence $\dim M \leq \dim K$. This implies $K$ is a cover ( using Lemma \ref{Lem4'}). So, to show all covers for a finite dimensional Lie superalgebra $L$ are isomorphic it is enough to show universal elements for $L$ are isomorphic. 

\begin{lemma}\label{Lem12}
Let $L$ be a Lie superalgebra of dimension $(m|n)$ and $0 \longrightarrow R \longrightarrow F \longrightarrow L \longrightarrow 0$ be a free presentation of $L$. Let
$$B=\frac{R}{[F, R]}, ~~~ C= \frac{F}{[F, R]}, ~~~ D=\frac{R \cap F'}{[F, R]},$$ be quotient Lie superalgebras then
\begin{enumerate}
\item $\dim D \leq \dim(\mbox{Ker}(\lambda))$, where $(T, \lambda) \in C(L)$ is the universal element.
\item $\mbox{Ker}(\psi)$ is a homomorphic image of $D$, for every element $(K, \psi) \in C(L)$.
\end{enumerate}
\end{lemma}
\begin{proof}
Clearly $[B, C]=0$, so $B \subseteq Z(C)$ is an graded ideal of $C$. Let $E$ be complementary superspace to $D$ in $B$, i.e., $B= D \oplus E$. As, $E \subseteq B \subseteq Z(C)$, so $E$ is a graded ideal of $C$. We intend to show $(C/E, \overline{\sigma}) \in C(L)$ where $\sigma: C \longrightarrow L$ is given by $\sigma(a+[F, R])= a+R $, for all $a \in F$. Evidently $\sigma $ is well defined onto map and also is a homogeneous linear map of even degree. Further for homogeneous elements $a+[F, R], b+[F, R]$, where $a, b \in F$, consider $\sigma([a+[F, R], b+[F, R]])= [a, b]+R=[a+R, b+R]$. It shows that $\sigma$ is a homomorphism. Now, $\sigma$ induces  the map \[ \overline{\sigma}:C/E \longrightarrow L \]
given by  \[~~~~~~\overline{\sigma}(x+E)=\sigma(x),\]
as $\sigma(B)=R$. So, $\overline{\sigma}$ is a surjective Lie superalgebra homomorphism. Now, $\mbox{Ker} (\overline{\sigma})= \{x+E ~|  ~ \overline{\sigma}(x+E)=R \}$. Here $x \in C$, say $x=b+[F, R]$ for $b \in F$. Finally, we get $\mbox{Ker} (\overline{\sigma})= \{b+[F, R]+E ~|~ b \in R\}=B/E$. 
Moreover, $$\mbox{Ker}(\overline{\sigma})=B/E \subseteq \frac{Z(C)+E}{E} \subseteq Z(\frac{C}{E})$$ and $$\mbox{Ker}(\overline{\sigma})=B/E=\frac{D \oplus E}{E} \subseteq \frac{C'+E}{E}=\left(\frac{C}{E}\right)^{'}.$$ Hence $C/E \in C(L)$. We have $(T, \lambda)$ is the universal element, by Lemma \ref{Lem11} $\dim(C/E) \leq \dim T$. So,

\begin{align*}
\dim(D) &= \dim \left( \frac{D \oplus E}{E}\right)=\dim(\frac{B}{E})\\
&=\dim(\frac{C}{E})- \dim(L)\\
& \leq \dim T-\dim L\\
&= \dim L + \dim \mbox{Ker}(\lambda)- \dim L= \dim \mbox{Ker}(\lambda),
\end{align*}
which proves our first assertation.
\par
Let $(K, \psi)$ be an element of $C(L)$. Now consider the canonical homomorphism $\pi$ from $F$ onto $L$, so $\mbox{Ker}(\pi)=R$. Since $F$ is free Lie superalgebra, there exists a homomorphism $\rho: F \longrightarrow K$ such that $\psi \circ \rho= \pi$ and we claim, $\rho$ is onto. Let $X$ be a $\mathbb{Z}_{2}$- graded set and the free Lie superalgebra $F$ is generated by $X$. Let us denote $ W =<\{\rho(x)| x \in X\}>$, then $K=<\{W+\mbox{Ker}(\psi)\}>$. Consider the Frattini subalgebra $\mathfrak{F}(K)$ of $K$ and let $\mbox{Ker}(\psi) \nsubseteq \mathfrak{F}(K)$. Then there exists a maximal subalgebra $M$ of $K$ and a non-zero homogeneous element $y \in \mbox{Ker}(\psi)\setminus M$. Hence $M + <y>= K$ and as $\mbox{Ker}(\psi) \subseteq Z(K) \cap K'$, so $K'=M'$. On the other hand, $y \in \mbox{Ker}(\psi) \subseteq M' \subseteq M$ which is a contradiction. Hence $\mbox{Ker}(\psi) \subseteq \mathfrak{F}(K)$ which implies $K=<W>$, i.e., $\rho$ is onto and $\rho(R)= \mbox{Ker} (\psi)$. So, $\rho[F, R]= [\rho(F), \rho(R)]=[K,\mbox{Ker}(\psi)]=0$. Now, let \[ \rho^{*}:\frac{F}{[F, R]} \longrightarrow K \] be the map
given by,  \[~~~~~~\rho^{*}(a+[F, R])= \rho(K).\] We have 
\begin{align*}
\rho^{*}(D) &= \rho^{*}(\frac{F' \cap R}{[F, R]})\\
&= \rho(F' \cap R)= \rho(F') \cap \rho(R)\\
&= K' \cap \mbox{Ker}(\psi)= \mbox{Ker}(\psi).
\end{align*}
\end{proof}

\begin{corollary}
Let $L$ be finite dimensional Lie superalgebra and $(T, \lambda)$ be the universal element, then $$
\mbox{Ker}(\lambda) \cong \frac{F' \cap R}{[F, R]}=D.$$
\end{corollary}

\begin{theorem}\label{Th4}
If $(T, \lambda)$ is a universal element of the Lie superalgebra $L$, then $T \sim \frac{F}{[F, R]}$.
\end{theorem}
\begin{proof}
Consider the map $\rho^{*}:\frac{F}{[F, R]} \longrightarrow T$ and by Lemma \ref{Lem12}, $\rho^{*}$ is surjective. Let $x+[R, F] \in \mbox{Ker}(\rho)^{*}$, for $x \in F$. Then
$\rho^{*}(x+[R, F])=\rho(x)=0$, which implies $\psi(\rho(x))=\pi(x)=0$, i.e., $x \in \mbox{Ker}(\pi)= R$. We have $\mbox{Ker}(\rho^{*}) \subseteq \frac{R}{[F, R]}$, then 
$$\mbox{Ker}(\rho^{*}) \cap  (\frac{F}{[F,R]})^{'}=\mbox{Ker}(\rho^{*}) \cap  (\frac{R}{[F,R]})\cap  (\frac{F}{[F,R]})^{'}$$ and hence $$
\mbox{Ker}(\rho^{*}) \cap  (\frac{F' \cap R}{[F,R]})= \mbox{Ker}(\rho^{*}) \cap D=0.$$ Hence, by Lemma \ref{Lem2'} we get $T \sim \frac{F}{[F, R]}$.
\end{proof}
Finally below is the main result of this section.
\begin{theorem}\label{Th5}
Let $(T_{1}, \lambda_{1})$ and $(T_{2}, \lambda_{2})$ be two universal element of a finite dimensional Lie superalgebra, then $T_{1} \cong T_{2}$.
\end{theorem}
\begin{proof}
By Theorem \ref{Th4}, we have $T_{1} \sim \frac{F}{[F, R]}$ and also $T_{2} \sim \frac{F}{[F, R]}$, which implies $T_{1} \sim T_{2}$. Further $\dim T_{1}=\dim T_{2}$. So, using Theorem \ref{Th3} $T_{1} \cong T_{2}$. 
\end{proof}

\end{document}